\documentclass{compositio}
\usepackage{mathrsfs}
\usepackage{amsmath}
\usepackage{amssymb}
\usepackage{verbatim}
\usepackage{xcolor}
\usepackage{graphicx}
\usepackage{textcomp}
\usepackage[all]{xy}
\numberwithin{equation}{section}

\theoremstyle{plain}
\newtheorem{theorem}{Theorem}[section]
\newtheorem{proposition}{Proposition}[section]
\newtheorem{corollary}{Corollary}[section]
\newtheorem{lemma}{Lemma}[section]

\theoremstyle{definition}
\newtheorem{definition}{Definition}[section]

\theoremstyle{remark}
\newtheorem{rem}{Remark}[section]

\begin{document}

\title{On the birational invariance of the balanced hyperbolic manifolds}

\author{Jixiang Fu}
\email{majxfu@fudan.edu.cn}
\address{School of Mathematical Sciences, Fudan University, Shanghai 200433, People's Republic of China}

\author{Hongjie Wang}
\email{21110180020@m.fudan.edu.cn}
\address{School of Mathematical Sciences, Fudan University, Shanghai 200433, People's Republic of China}

\author{Jingcao Wu}
\email{wujincao@shufe.edu.cn}
\address{School of Mathematics, Shanghai University of Finance and Economics, Shanghai 200433, People's Republic of China}

\thanks{Fu is supported by NFSC, grant 12141104; Wu is supported by NFSC, grant 12201381.}

\begin{abstract}
In this paper, we discuss the birational invariance of the class of balanced hyperbolic manifolds. 
\end{abstract}

\maketitle

\section{Introduction}
\label{sec:introduction}
In his celebrated paper \cite{Gro91}, M. Gromov introduces an important notion called the K\"{a}hler hyperbolicity. It is pinched between the real hyperbolicity and the Kobayashi hyperbolicity \cite{Kob98}, and helps to settle the K\"{a}hler case of the Chern conjecture \cite{Gro91}. After that, it leads to fruitful applications and improvements such as \cite{CX01,CY18,Eys97,Hit00,Kol95,McM00} and so on.

However, since the class of K\"{a}hler manifolds in general is not invariant under the birational transform, it would be desirable to have a birational variant of the K\"{a}hler hyperbolicity developed. It is an open problem posed by J. Koll\'{a}r in \cite{Kol95}. Koll\'{a}r suggests to require Gromov's condition for a degenerate K\"{a}hler form, and \cite{BCDT24,BDET24} introduce the weakly K\"{a}hler hyperbolicity by asking the cohomology class to be nef and big rather than K\"{a}hler. Weakly K\"{a}hler hyperbolic manifolds possess many key features as K\"{a}hler hyperbolic manifolds, and are invariant under the birational transform. Whereas in this paper, we are trying to investigate a more general situation, namely the balanced hyperbolicity.

More precisely, let $X$ be a compact complex manifold of dimension $n$. A Hermitian metric $\omega$ on $X$ is called balanced if $d\omega^{n-1}=0$. $X$ is called a balanced manifold if it possesses a balanced metric. Obviously a K\"{a}hler metric must be balanced, but there do exist non-K\"{a}hler balanced metrics. Hence a balanced form is regarded as a generalization of  K\"{a}hler form. A celebrated theorem in \cite{AB95} asserts that the class of compact balanced manifolds is invariant under the smooth modification, which directly inspires this paper.

Let $\pi:\tilde{X}\rightarrow X$ be the universal cover, and fix a Riemannian metric $g$ on $X$. Recall that a $k$-form $\alpha$ on $X$ is called $\tilde{d}$-bounded, if there exists a $(k-1)$-form $\beta$ on $\tilde{X}$ such that $\pi^{\ast}\alpha=d\beta$ and $\sup_{\tilde{X}}\|\beta\|_{\pi^{\ast}g}<\infty$. Note since $X$ is compact, this notion is actually independent of the choice of $g$. Moreover, an easy argument (c.f. Lemma \ref{l23}) shows that we can even talk about the $\tilde{d}$-boundedness for a de Rham cohomology class, in which case we also call it hyperbolic. 

Then we say that a Hermitian metric $\omega$ on $X$ is {\it balanced hyperbolic} if $\omega$ is balanced and $\omega^{n-1}$ is $\tilde{d}$-bounded. It is notable that the balanced hyperbolicity was first introduced in \cite{MP22,MP23}, and we will continue their discussion in a wider range. Observe that if $\omega$ is balanced, $[\omega^{n-1}]$ must be nef and big as an $(n-1,n-1)$-class which is defined in Section 2. It allows us to talk about more degenerate cases. We say that a real smooth $(1,1)$-form $\omega$ is {\it semi-balanced hyperbolic} if $\omega^{n-1}$ is $d$-closed, non-negative, strictly positive on a Zariski open set, and $\tilde{d}$-bounded. Fix a positive integer $k\mid (n-1)$, say $kt=n-1$. We say that a real smooth $(k,k)$-form $\beta$ is {\it weakly balanced $k$-hyperbolic} if $\beta$ is $\tilde{d}$-bounded and $[\beta^t]$ is nef and big. Finally, we say that a balanced manifold $X$ is a (weakly or semi-)balanced $k$-hyperbolic manifold, if there exists a (weakly or semi-)balanced hyperbolic form on it. 

In a similar atmosphere, we can also talk about the $k$-hyperbolicity for the K\"{a}hler case. We say that a Hermitian metric $\omega$ on $X$ is {\it K\"{a}hler $k$-hyperbolic} if $\omega$ is K\"{a}hler and $\omega^{k}$ is $\tilde{d}$-bounded. We say that a real smooth $(1,1)$-form $\omega$ is {\it semi-K\"{a}hler $k$-hyperbolic} if $\omega$ is $d$-closed, non-negative, strictly positive on a Zariski open set and $\omega^{k}$ is $\tilde{d}$-bounded. We say that a real smooth $(1,1)$-form $\omega$ is {\it weakly K\"{a}hler $k$-hyperbolic} if $\omega$ is $d$-closed, $[\omega]$ is nef and big, and $\omega^{k}$ is $\tilde{d}$-bounded. Finally, we say a K\"ahler manifold $X$ is a (weakly or semi-) K\"{a}hler $k$-hyperbolic manifold, if there exists a (weakly or semi-) K\"{a}hler $k$-hyperbolic form on it. In Sect. \ref{sec:hyper} we will discuss the relationship among these hyperbolicities.

Now we can state our results concerning the behavior of the balanced hyperbolicity under the birational transform.  Recall a birational transform $f:\hat{X}\dashrightarrow X$ between compact complex manifolds is a meromorphic map such that, for suitable analytic subvarieties $Y$ (called the center) of $X$ and $E$ (called the exceptional set) of $\hat{X}$, $f|_{\hat{X}\setminus E}:\hat{X}\setminus E\rightarrow X\setminus Y$ is a biholomorphic map. Without loss of generality, we usually ask $\textrm{codim} E\geqslant1$, $\textrm{codim}Y\geqslant2$ and $\dim Y\leqslant\dim E$. In particular, when $\textrm{codim}E=1$, we call it a divisorial contraction; when $\textrm{codim}E\geqslant2$, we call it a small contraction; when $f$ is moreover a holomorphic map, we call it a smooth modification. In the last situation, a standard argument shows that we must have $\textrm{codim}E=1$ if $f$ is not a biholomorphic map itself.

For a given birational transform $f:\hat{X}\dashrightarrow X$, let
\[
\Gamma_{f}:=\{(x,f(x))\in \hat{X}\times X\mid x\in\hat{X}\setminus E\},
\]
and let $\mu:W\rightarrow\bar{\Gamma}_{f}$ be the desingularization of the closure $\bar{\Gamma}_{f}$. Let $q_{1}:\bar{\Gamma}_{f}\rightarrow\hat{X}$ and $q_{2}:\bar{\Gamma}_{f}\rightarrow X$ be the natural projections. Then $p_{1}:=q_{1}\circ\mu:W\rightarrow\hat{X}$ and $p_{2}:=q_{2}\circ\mu:W\rightarrow X$ are modifications. Denote by $V$ the exceptional divisor of $\mu$. Clearly we have $p_{1}(V)\subseteq E$ and $p_{2}(V)\subseteq Y$.

\begin{theorem}\label{t11}
Let $f:\hat{X}\dashrightarrow X$ be a birational transform between compact complex manifolds of dimension $n$.
\begin{enumerate}
\item[(1)] Assume that $f$ is a divisorial contraction, and the higher homotopy groups $\pi_{i}(W)=0$ for $2\leqslant i\leqslant 2n-3$. If $\hat{X}$ is balanced hyperbolic, then $X$ is also balanced hyperbolic.
\item[(2)] Assume that $f$ is a small contraction, and the higher homotopy groups $\pi_{i}(W)=0$ for $2\leqslant i\leqslant 2n-3$. Then $\hat{X}$ is balanced hyperbolic if and only if $X$ is balanced hyperbolic.
\end{enumerate}
\end{theorem}

Next, we investigate the K\"{a}hler case.

\begin{theorem}\label{t12}
Let $f:\hat{X}\dashrightarrow X$ be a birational transform between compact K\"{a}hler manifolds of dimension $n$. Fix a positive integer $k\mid(n-1)$. Assume that the higher homotopy groups $\pi_{i}(W)=0$ for $2\leqslant i\leqslant 2k-1$. 

If $\hat{X}$ is weakly K\"{a}hler $k$-hyperbolic, then $X$ is weakly balanced $k$-hyperbolic. Conversely, if $X$ is weakly K\"{a}hler $k$-hyperbolic, then $\hat{X}$ is weakly balanced $k$-hyperbolic.
\end{theorem}

When $f$ is a smooth modification, we have
\begin{theorem}\label{t13}
Let $f:\hat{X}\rightarrow X$ be a smooth modification between compact complex manifolds of dimension $n$.
\begin{enumerate}
\item[(1)] Assume that the higher homotopy groups $\pi_{i}(\hat{X})=0$ for $2\leqslant i\leqslant 2n-3$. If $\hat{X}$ is balanced hyperbolic, then $X$ is also balanced hyperbolic.
\item[(2)] If $X$ is semi-balanced hyperbolic (resp. semi-K\"{a}hler $k$-hyperbolic), then $\hat{X}$ is also semi-balanced hyperbolic (resp. semi-K\"{a}hler $k$-hyperbolic).
\item[(3)] Fix a positive integer $k\mid(n-1)$. Assume that both $\hat{X}$ and $X$ are K\"{a}hler, and the higher homotopy groups $\pi_{i}(\hat{X})=0$ for $2\leqslant i\leqslant 2k-1$. If $\hat{X}$ is weakly balanced $k$-hyperbolic, then $X$ is also weakly balanced $k$-hyperbolic.
\end{enumerate}
\end{theorem}

\begin{rem}
\cite{BCDT24} develops a beautiful topological technique to show that for a birational transform $f:\hat{X}\dashrightarrow X$, $\hat{X}$ is weakly K\"{a}hler 1-hyperbolic if and only if $X$ is weakly K\"{a}hler 1-hyperbolic. Our theorems owe it a lot. In particular, Theorem \ref{t12} extends their result to the $k$-hyperbolicity.
\end{rem}

In the ending section, we attempt to get rid of the vanishing condition about the higher homotopy groups. It is possible at least when $\dim Y=0$, in which case we call $f$ a {\it contraction to points}.

\begin{theorem}\label{t14}
Let $f:\hat{X}\dashrightarrow X$ be a contraction to points between compact complex manifolds of dimension $n$. If $\hat{X}$ is balanced hyperbolic, then $X$ is also balanced hyperbolic.
\end{theorem}

It naturally leads to the following generalization. Note that we denote by $V^{k}_{\textrm{hyp}}$ the collection of the hyperbolic $k$-cohomology classes, i.e., cohomology classes with hyperbolic representatives.
\begin{theorem}\label{t15}
Let $f:\hat{X}\dashrightarrow X$ be a birational transform between compact complex manifolds of dimension $n$, such that the dimension of center $Y$ is $n-k-1$. Then 
\[
(q_{2}|_{\Gamma_{f}})^{\ast}(V^{2k}_{\textrm{hyp}}(X))=V^{2k}_{\textrm{hyp}}(\Gamma_{f})\quad\textrm{and}\quad(q_{2}|_{\Gamma_{f}})_{\ast}(V^{2k}_{\textrm{hyp}}(\Gamma_{f}))=V^{2k}_{\textrm{hyp}}(X).
\]

Moreover, if $k\mid(n-1)$, $\hat{X}$ is weakly K\"ahler $k$-hyperbolic and $X$ is a K\"ahler manifold, then $X$ is weakly balanced $k$-hyperbolic.
\end{theorem}

All of the theorems above show that these hyperbolicities surely possess certain invariance under the birational transform, and are also good candidates to solve Koll\'{a}r's problem.

\section{Preliminary}
\label{sec:preliminary}
Let $X$ be a compact complex manifold of dimension $n$. 

\subsection{Positive cones}
We first recall a few positive cones in the Bott--Chern cohomology groups 
\[
H^{1,1}_{\textrm{BC}}(X,\mathbb{R})\quad\textrm{and}\quad H^{n-1,n-1}_{\textrm{BC}}(X,\mathbb{R}).
\] 
Note a $(1,1)$-class is usually called a divisor class, while an $(n-1,n-1)$-class a curve class.

\begin{definition}\label{d21}
\begin{enumerate}
\item Let 
\[
\mathcal{E}(X):=\{[T]\in H^{1,1}_{\textrm{BC}}(X,\mathbb{R})\mid T\textrm{ is a positive }d\textrm{-closed }(1,1)\textrm{-current on }X\},
\]
and
\[
\mathcal{N}(X):=\{[T]\in H^{n-1,n-1}_{\textrm{BC}}(X,\mathbb{R})\mid T\textrm{ is a positive }d\textrm{-closed }(n-1,n-1)\textrm{-current on }X\}.
\]
They are obviously closed cones in $H^{1,1}_{\textrm{BC}}(X,\mathbb{R})$ and $H^{n-1,n-1}_{\textrm{BC}}(X,\mathbb{R})$ respectively, which are called the pseudo-effective divisor cone and the pseudo-effective curve cone.
\item $\mathcal{E}(X)^{\circ}$ (resp. $\mathcal{N}(X)^{\circ}$) is called the big divisor (resp. curve) cone.
\item Let
\[
\mathcal{K}(X):=\{[\omega]\in H^{1,1}_{\textrm{BC}}(X,\mathbb{R})\mid\omega\textrm{ is a smooth, strictly positive }d\textrm{-closed }(1,1)\textrm{-form on }X\}, 
\]
and
\[
\begin{split}
\mathcal{B}(X):=&\{[\Omega]\in H^{n-1,n-1}_{\textrm{BC}}(X,\mathbb{R})\mid\Omega\textrm{ is a smooth, strictly positive }d\textrm{-closed } \\
&(n-1,n-1)\textrm{-form on }X\}.
\end{split}
\]
They are obviously open cones in $H^{1,1}_{\textrm{BC}}(X,\mathbb{R})$ and $H^{n-1,n-1}_{\textrm{BC}}(X,\mathbb{R})$ respectively, which are called the K\"{a}hler cone and the balanced cone of $X$.
\item Fix a Hermitian metric $\sigma$ on $X$. Let
\[
\begin{split}
\overline{\mathcal{K}(X)}:=&\{[\omega]\in H^{1,1}_{\textrm{BC}}(X,\mathbb{R})\mid\textrm{For any }\varepsilon>0, \textrm{there exists a smooth representative }\omega_{\varepsilon}\in[\omega] \\
&\textrm{ such that }\omega_{\varepsilon}\geqslant-\varepsilon\sigma\},
\end{split} 
\]
and
\[
\begin{split}
\overline{\mathcal{B}(X)}:=&\{[\Omega]\in H^{n-1,n-1}_{\textrm{BC}}(X,\mathbb{R})\mid\textrm{For any }\varepsilon>0, \textrm{there exists a smooth representative }\\
&\Omega_{\varepsilon}\in[\Omega]\textrm{ such that }\Omega_{\varepsilon}\geqslant-\varepsilon\sigma^{n-1}\}.
\end{split}
\]
$\overline{\mathcal{K}(X)}$ (resp. $\overline{\mathcal{B}(X)}$) is called the nef divisor (resp. curve) cone of $X$. Clearly, $\overline{\mathcal{K}(X)}\subseteq\mathcal{E}(X)$ and $\overline{\mathcal{B}(X)}\subseteq\mathcal{N}(X)$. When $\mathcal{K}(X)$ (resp. $\mathcal{B}(X)$) is not empty, $\overline{\mathcal{K}(X)}$ (resp. $\overline{\mathcal{B}(X)}$) is its closure.
\end{enumerate}
\end{definition}
A $(1,1)$-class is called pseudo-effective (resp. big, nef,...) if it belongs to $\mathcal{E}(X)$ (resp. $\mathcal{E}(X)^{\circ}$, $\overline{\mathcal{K}(X)}$,...). The positivity for the $(n-1,n-1)$-classes is similarly defined.

We should recall some basic properties of the balanced cone in \cite{FX14}. Remember that a Hermitian metric $\omega$ such that $d\omega^{n-1}=0$ is called a balanced metric. We say $X$ is a balanced manifold, if there exists a balanced metric on it. In fact, the existence of a balanced metric $\omega$ is equivalent to the existence of a $d$-closed strictly positive $(n-1,n-1)$-form $\Omega$ by \cite{Mic83}.
Hence, for convenience, such $\Omega$ will also be called a balanced metric, which leads to the definition of the balanced cone $\mathcal{B}(X)$ above. In this paper, we will always use $\omega,\hat{\omega}$ to denote the $(1,1)$-form and the capital letters such as $\Omega,\hat{\Omega}$ to denote the $(n-1,n-1)$-form.

The relationship between the balanced cone and the divisor cone is subtle. We denote by $A^{p,q}(X)$ the space of the smooth $\mathbb{C}$-valued $(p,q)$-forms and by $A^{p,q}_{\mathbb{R}}(X)$ the space of the smooth $\mathbb{R}$-valued $(p,q)$-forms. Define
\[
V^{p,p}(X,\mathbb{R}):=\frac{\{\phi\in A^{p,p}_{\mathbb{R}}(X)|\partial\bar{\partial}\phi=0\}}{\{\partial A^{p-1,p}(X)+\bar{\partial}A^{p,p-1}(X)\}\cap A^{p,p}_{\mathbb{R}}(X)}.
\]
It is well-known that we can replace $A^{p,p}$ by the space of $(p,p)$-currents in the above definition. If we denote $\mathcal{E}_{\partial\bar{\partial}}(X)\subseteq V^{1,1}(X,\mathbb{R})$ the convex cone generated by $\partial\bar{\partial}$-closed positive $(1,1)$-currents, it is proved in \cite{FX14} (Lemma 3.3 \& Remark 3.4) that
\begin{lemma}\label{l21}
If $X$ is a compact balanced manifold, then $\mathcal{E}_{\partial\bar{\partial}}(X)^{\vee}=\overline{\mathcal{B}(X)}$.
\end{lemma}
Then combining the above lemma with a celebrated description of the existence of balanced metric \cite{Mic83}, we reformulate it as 
 
\begin{proposition}\label{p21}
Let $X$ be a compact balanced manifold of dimension $n$, and let $[\Omega]$ be a real $(n-1,n-1)$-class on $X$. Then the following two statements are equivalent:
\begin{enumerate}
\item[(a)] $[\Omega]$ is a balanced class;
\item[(b)] for any positive $\partial\bar{\partial}$-closed $(1,1)$-current $T$, $\int_{X}[\Omega]\wedge T\geqslant0$ and $\int_{X}[\Omega]\wedge T=0$ if and only if $T=0$.
\end{enumerate}
\begin{proof}
(b)$\Rightarrow$(a) is nothing but \cite{FX14}, Lemma 3.3.

(a)$\Rightarrow$(b) is due to \cite{Mic83}, Proposition 4.2. 
\end{proof}
\end{proposition}

\begin{rem}
If $V^{1,1}(X,\mathbb{R})=H^{1,1}_{\textrm{BC}}(X,\mathbb{R})$, for example, when $X$ is a K\"{a}hler manifold, Lemma \ref{l21} is reformulated as
\begin{equation}\label{e21}
\mathcal{E}(X)^{\vee}=\overline{\mathcal{B}(X)}.
\end{equation}
\end{rem}

\subsection{The hyperbolicity}
\label{sec:hyper}
With the preparations above, we are able to define various hyperbolicities mentioned in Introduction.

\begin{definition}\label{d22}
For $k=1,...,n$, we say that a Hermitian metric $\omega$ is balanced $k$-hyperbolic, if it is balanced and $\omega^{k}$ is $\tilde{d}$-bounded. We say that a real smooth $(1,1)$-form $\omega$ is semi-balanced $k$-hyperbolic, if $\omega^{n-1}$ is $d$-closed, non-negative and strictly positive on a Zariski open set and $\omega^{k}$ is $\tilde{d}$-bounded. Fix a positive integer $k\mid (n-1)$, say $kt=n-1$. We say that a real smooth $(k,k)$-form $\beta$ is weakly balanced $k$-hyperbolic, if $\beta$ is $\tilde{d}$-bounded, and $[\beta^t]$ is nef and big.

Let $X$ be a compact balanced manifold. We say $X$ is a (weakly or semi-) balanced $k$-hyperbolic manifold, if there exists a (weakly or semi-) balanced $k$-hyperbolic form on it. 
\end{definition}

Remember that a Hermitian metric $\omega$ is K\"{a}hler hyperbolic \cite{Gro91} if it is K\"{a}hler and $\tilde{d}$-bounded; whereas a real smooth $(1,1)$-form $\omega$ is semi-K\"{a}hler hyperbolic \cite{Kol95}, if $\omega$ is $d$-closed, non-negative, strictly positive on a Zariski open set and $\tilde{d}$-bounded; whereas a real smooth $(1,1)$-form $\omega$ is weakly K\"{a}hler hyperbolic \cite{BDET24}, if it is $d$-closed, $\tilde{d}$-bounded, and $[\omega]$ is nef and big. 

We can also generalize these notions to the $k$-hyperbolicity.
\begin{definition}\label{d23}
For $k=1,...,n$, we say that a Hermitian metric $\omega$ is K\"{a}hler $k$-hyperbolic, if it is K\"{a}hler and $\omega^{k}$ is $\tilde{d}$-bounded. We say that a real smooth $(1,1)$-form $\omega$ is semi-K\"{a}hler $k$-hyperbolic, if $\omega$ is $d$-closed, non-negative and strictly positive on a Zariski open set and $\omega^{k}$ is $\tilde{d}$-bounded. We say that a real smooth $(1,1)$-form $\omega$ is weakly K\"{a}hler $k$-hyperbolic, if $\omega$ is $d$-closed, $[\omega]$ is nef and big, and $\omega^{k}$ is $\tilde{d}$-bounded.

Let $X$ be a compact K\"{a}hler manifold. We say that $X$ is a (weakly or semi-) K\"{a}hler $k$-hyperbolic manifold, if there exists a (weakly or semi-) K\"{a}hler $k$-hyperbolic form on it.
\end{definition}

Next we make a discussion about the relationship among these hyperbolicities. Clearly, Gromov's (weakly or semi-)K\"{a}hler hyperbolicity is nothing but (weakly or semi-)K\"{a}hler $1$-hyperbolicity in our language. For a balanced $k$-hyperbolic metric $\omega$, $\pi^{\ast}d\omega^{k}=d\pi^{\ast}\omega^{k}=0$ since $\pi^{\ast}\omega^{k}$ is $d$-exact. However, it implies that $d\omega^{k}=0$ because $\tilde{X}\rightarrow X$ is locally biholomorphic. Moreover, when $k<n-1$, $d\omega^{k}=0$ indicates that $d\omega=0$ by standard multi-linear algebra. All in all, we conclude that a balanced $k$-hyperbolic manifold with $k<n-1$ must be K\"{a}hler $k$-hyperbolic. It also happens when $\omega$ is semi-balanced. We first obtain $d\omega=0$ on the Zariski open set where $\omega$ is strictly positive, then deduce that $d\omega=0$ on the whole space by smoothness. Hence the (semi-)balanced $k$-hyperbolicity only makes sense when $k=n-1$, in which case we simply call it {\it (semi-)balanced hyperbolicity}. Clearly this argument fails when $\omega$ further degenerates.   

By definition, a K\"{a}hler $k$-hyperbolic manifold must be both semi-K\"{a}hler $k$-hyperbolic and balanced $k$-hyperbolic with $k=1,...,n$. A semi-K\"{a}hler $k$-hyperbolic manifold is clearly both weakly K\"{a}hler $k$-hyperbolic and semi-balanced $k$-hyperbolic. A balanced hyperbolic manifold must be semi-balanced hyperbolic. Also we have 

\begin{lemma}\label{l22}
A weakly K\"{a}hler $k$-hyperbolic manifold must be weakly balanced $k$-hyperbolic for $k\mid (n-1)$.
\begin{proof}
Let us recall two crucial facts in \cite{LX16}. Let $[\alpha]\in\mathcal{N}(X)$, and define the volume of $[\alpha]$ to be
\[
\widehat{\textrm{vol}}([\alpha])=\inf_{[A] \textrm{ big and nef (1,1)-class}}\big(\frac{[A]\cdot[\alpha]}{\textrm{vol}([A])^{1/n}}\big)^{\frac{n}{n-1}}.
\]
Then \cite{LX16} indicates that
\begin{theorem}[(c.f. \cite{LX16}, Theorem 5.2)]\label{t21}

\begin{enumerate}
\item[(1)] $\widehat{\textrm{vol}}$ is positive precisely for the big classes.
\item[(2)] For any big and nef $(1,1)$-class $[A]$, we have $\widehat{\textrm{vol}}([A^{n-1}])=\textrm{vol}([A])$.
\end{enumerate}
\end{theorem}
Although \cite{LX16} is formulated for projective manifolds, it is also carefully explained in \cite{LX16}, Sect. 2.4 that everything extends smoothly to the K\"{a}hler case. 

Now if $X$ is a weakly K\"{a}hler $k$-hyperbolic manifold, then there exists a nef and big $(1,1)$-class $[\omega]$ such that $\omega^{k}$ is $\tilde{d}$-bounded. So $\widehat{\textrm{vol}}([\omega^{n-1}])=\textrm{vol}([\omega])>0$. The inequality is due to \cite{DP04}, Theorem 0.5. Therefore $[\omega^{n-1}]$ is a nef and big $(n-1,n-1)$-class. It exactly implies that $X$ is weakly balanced $k$-hyperbolic.
\end{proof}
\end{lemma}

For a balanced metric $\omega$, $\omega^{n-1}$ is certainly positive hence $[\omega^{n-1}]$ is nef. It means that $\mathcal{B}(X)$ is an open subcone of $\mathcal{N}(X)$, so $\mathcal{B}(X)\subseteq\mathcal{N}(X)^{\circ}$. Therefore $[\omega^{n-1}]$ is also big. It means that a balanced hyperbolic metric must be weakly balanced hyperbolic. 

In the end, let $\pi:\tilde{X}\rightarrow X$ be the universal cover, and fix a Riemannnian metric $g$ on $X$. If $\omega^{k}$ is $\tilde{d}$-bounded, then there exists a $(2k-1)$-form $\eta$ on $\tilde{X}$ such that $\pi^{\ast}\omega^{k}=d\eta$ and $\eta$ is bounded with respect to $\pi^{\ast}g$. Then for any positive integer $t$,
\[
\pi^{\ast}\omega^{kt}=d\eta\wedge\pi^{\ast}\omega^{k(t-1)}=d(\eta\wedge\pi^{\ast}\omega^{k(t-1)}).
\] 
Since $X$ is compact, $\omega$ is bounded with respect to $g$. Hence $\pi^{\ast}\omega^{k(t-1)}$ as well as $\eta\wedge\pi^{\ast}\omega^{k(t-1)}$ is bounded with respect to $\pi^{\ast}g$. It exactly means that the $k$-hyperbolicity implies the $kt$-hyperbolicity for any positive integers $k$ and $t$.

We summarize the relationship among these notions as follows: for any positive integers $k$ and $t$, we have 
\[
k\textrm{-hyperbolicity}\Rightarrow kt\textrm{-hyperbolicity}
\]
and for any positive integer $k$, we have
\[
\xymatrix{
\textrm{K\"{a}hler }k\textrm{-hyperbolicity} \ar@{=>}[d] \ar@<.5ex>[r] & \textrm{balanced }k\textrm{-hyperbolicity} \ar@<.5ex>[l]^{k<n-1} \ar@{=>}[d]  \ar@{=>}@<14.5ex>@/^/[dd]        \\
\textrm{semi-K\"{a}hler }k\textrm{-hyperbolicity} \ar@{=>}[d] \ar@<.5ex>[r]  & \textrm{semi-balanced }k\textrm{-hyperbolicity} \ar@<.5ex>[l]^{k<n-1}     \\ 
\textrm{weakly K\"{a}hler }k\textrm{-hyperbolicity} \ar@{=>}[r]^{k\mid (n-1)}  & \textrm{weakly balanced }k\textrm{-hyperbolicity}       
}
\]
It is worthwhile to point out that there exist weakly K\"{a}hler 1-hyperbolic manifolds which are not K\"{a}hler 1-hyperbolic manifolds provided by \cite{BCDT24,BDET24}. Moreover, the connected sums 
\[
\#_{k}(S^{3}\times S^{3})\textrm{ with }k\geqslant2
\] 
are balanced manifolds as is shown in \cite{FLY12}. A direct computation implies that 
\[
H^{4}_{\textrm{dR}}(\#_{k}(S^{3}\times S^{3}),\mathbb{R})=0.
\] 
So they cannot be compact K\"{a}hler manifolds. On the other hand, any balanced metric $\omega$ gives a zero class $[\omega^{2}]$. It means that $\omega$ is balanced hyperbolic. Therefore there exist non-K\"{a}hler (weakly or semi-)balanced $2$-hyperbolic manifolds, and the horizontal implications in the diagram are strict except the mentioned cases. We are willing to know more information about the vertical inclusions. For example, must a semi-balanced hyperbolic manifold be weakly balanced $(n-1)$-hyperbolic?

We also  have the following interesting description for the fundamental group of a weakly balanced $k$-hyperbolic manifold, though it is not really involved in our main theorems.
\begin{proposition}\label{p22}
Let $X$ be a compact K\"{a}hler manifold. If $k\mid (n-1)$, and $X$ is weakly balanced $k$-hyperbolic, then $\pi_{1}(X)$ is not amenable.
\begin{proof}
Let $\pi:\tilde{X}\rightarrow X$ be the universal cover. Denote by $H^{i}_{\beta}(\tilde{X})$ the de Rham cohomology based on differential forms $\alpha$ such that $\alpha$ and $d\alpha$ are uniformly bounded. Now assume that $\pi_{1}(X)$ is amenable. Due to \cite{ABW92}, the pull-back homomorphism $H^{i}(X)\rightarrow H^{i}_{\beta}(\tilde{X})$ is injective. But this is impossible. 

In fact, the pull-back of the weakly balanced $k$-hyperbolic class $[\gamma]$ is zero in $H^{2k}_{\beta}(\tilde{X})$. By hypothesis $k\mid(n-1)$, say $kt=n-1$. So $\pi^{\ast}[\gamma^{t}]$ is also a zero class in $H^{2n-2}_{\beta}(\tilde{X})$. However, $[\gamma^{t}]$ is big hence never a zero class in $H^{2n-2}(X)$. Otherwise, $[A]\cdot[\gamma^{t}]=0$ for any $(1,1)$-class $[A]$. Hence $\widehat{\textrm{vol}}([\gamma^{t}])=0$, which leads to a contradiction to Theorem \ref{t21}. Therefore, $\pi_{1}(X)$ is not amenable.
\end{proof}
\end{proposition}

In the end of this section, we provide the following property that is frequently used when verifying the $\tilde{d}$-boundedness.
\begin{lemma}\label{l23}
Let $f:\hat{X}\rightarrow X$ be a smooth modification between compact complex manifolds. Fix Riemannian metrics $\hat{g}_{X}$ and $g_{X}$ on $\hat{X}$ and $X$ respectively, and suppose that $\omega$ is a $\tilde{d}$-bounded $k$-form on $X$. Let $\xi$ (resp. $\theta$) be an arbitrary smooth $(k-1)$-form (resp. $(k-2)$-form) on $X$. Then $f^{\ast}\omega$, $\omega+d\xi$ and $\omega+\partial\bar{\partial}\theta$ are all $\tilde{d}$-bounded.
\begin{proof}
Let $\pi:\tilde{X}\rightarrow X$ be the universal cover. Since $f$ is birational,
\[
f_{\ast}:\pi_{1}(\hat{X})\rightarrow\pi_{1}(X)
\]
is an isomorphism (see \cite{BP21}, Proposition 2.3). Therefore $\chi:Z:=f^{\ast}\tilde{X}\rightarrow\hat{X}$ gives the universal cover of $\hat{X}$. Denote by $h$ the natural morphism $Z\rightarrow\tilde{X}$, and we have the following commutative diagram.

\begin{equation*}
\xymatrix@C=2pc@R=2pc{
Z\ar[r]^{h}\ar[d]^{\chi}& \tilde{X}\ar[d]^{\pi}\\
\hat{X}\ar[r]^{f} & X
}
\end{equation*}
Since $\omega$ is $\tilde{d}$-bounded, there exists a bounded $(k-1)$-form $\eta$ on $\tilde{X}$ such that $\pi^{\ast}\omega=d\eta$. Now by \cite{BDET24}, Lemma 2.28, $h^{\ast}\eta$ is also bounded with respect to $h^{\ast}\hat{g}_{X}$. Clearly $\chi^{\ast}f^{\ast}\omega=dh^{\ast}\eta$, so $f^{\ast}\omega$ is $\tilde{d}$-bounded. 

The $\tilde{d}$-boundedness of $\omega+d\xi$ and $\omega+\partial\bar{\partial}\theta$ are much easier. Since $X$ is compact, $\xi$ and $\bar{\partial}\theta$ are bounded with respect to $g_{X}$. Then $\eta+\pi^{\ast}\xi$ and $\eta+\pi^{\ast}\bar{\partial}\theta$ are also bounded with respect to $\pi^{\ast}g_{X}$. Note $\pi^{\ast}(\omega+d\xi)=d(\eta+\pi^{\ast}\xi)$ and $\pi^{\ast}(\omega+\partial\bar{\partial}\theta)=d(\eta+\pi^{\ast}\bar{\partial}\theta)$. Thus $\omega+d\xi$ and $\omega+\partial\bar{\partial}\theta$ are $\tilde{d}$-bounded. 
\end{proof}   
\end{lemma}

\subsection{Pull-back and push-forward}
In this section we list several formulas concerning the pull-backs and push-forwards involved in a birational transform, which are frequently applied in the later part without specifying. They are quite standard to experts, but we would like to provide a simple proof here for readers' benefits.

Let $f:\hat{X}\dashrightarrow X$ be a birational transform between compact complex manifolds of dimension $n$ with the exceptional set $E$ and the center $Y$. Let 
\[
\Gamma_{f}:=\{(x,f(x))\in\hat{X}\times X\mid x\in\hat{X}\setminus E\},
\] 
and let $\mu:W\rightarrow\bar{\Gamma}_{f}$ be the desingularization of the closure $\bar{\Gamma}_{f}$. Let $q_{1}:\bar{\Gamma}_{f}\rightarrow\hat{X}$ and $q_{2}:\bar{\Gamma}_{f}\rightarrow X$ be the natural projections. Let $p_{1}:=q_{1}\circ\mu$ and $p_{2}:=q_{2}\circ\mu$. Denote by $V$ the exceptional divisor in $W$ such that $p_{1}(V)\subseteq E$ and $p_{2}(V)\subseteq Y$, and denote $l_{1}=q_{1}|_{\Gamma_{f}}$ and $l_{2}=q_{2}|_{\Gamma_{f}}$. 

Note $\Gamma_{f}$ is an open manifold, the pull-backs and push-forwards induced by $l_1:\Gamma_f\rightarrow\hat{X}$ and $l_2:\Gamma_f\rightarrow X$ are well-defined. We have

\begin{lemma}\label{l24}
Let $\alpha$ (resp. $\beta$) be a smooth $k$-form on $X$ (resp. $\hat{X}$). Then
\[
(p_{1})_{\ast}(p_{2})^{\ast}\alpha=(l_{1})_{\ast}(l_{2})^{\ast}\alpha\quad\textrm{and}\quad(p_{2})_{\ast}(p_{1})^{\ast}\beta=(l_{2})_{\ast}(l_{1})^{\ast}\beta.
\]
\begin{proof}
Let $\chi$ be a test form on $\hat{X}$. Then
\begin{equation}\label{e22}
\begin{split}
\int_{\hat{X}}(p_{1})_{\ast}(p_{2})^{\ast}\alpha\wedge\chi&=\int_{W}(p_{2})^{\ast}\alpha\wedge(p_{1})^{\ast}\chi \\
&=\int_{W\setminus V}(p_{2})^{\ast}\alpha\wedge(p_{1})^{\ast}\chi \\
&=\int_{\Gamma_{f}}(l_{2})^{\ast}\alpha\wedge(l_{1})^{\ast}\chi \\
&=\int_{\hat{X}}(l_{1})_{\ast}(l_{2})^{\ast}\alpha\wedge\chi.
\end{split}
\end{equation}
The second equality is due to the facts that both of $(p_{2})^{\ast}\alpha$ and $(p_{1})^{\ast}\chi$ are smooth which implies the finiteness of the integral, and $V$ is of measure zero. (\ref{e22}) exactly indicates that $(p_{1})_{\ast}(p_{2})^{\ast}\alpha=(l_{1})_{\ast}(l_{2})^{\ast}\alpha$. The other one is similar.
\end{proof}
\end{lemma}

\begin{lemma}\label{l25}
Let $\alpha$ (resp. $\beta$) be a smooth $k$-form on $\hat{X}$ (resp. $X$). Then
\[
(p_{2})_{\ast}(p_{2})^{\ast}\alpha=\alpha,\quad(l_{2})_{\ast}(l_{2})^{\ast}\alpha=\alpha
\]
and
\[
(p_{1})_{\ast}(p_{1})^{\ast}\beta=\beta,\quad(l_{1})_{\ast}(l_{1})^{\ast}\beta=\beta.
\]
\begin{proof}
Let $\chi$ be a test form on $X$. Then 
\begin{equation}\label{e23}
\begin{split}
\int_{X}(p_{2})_{\ast}(p_{2})^{\ast}\alpha\wedge\chi&=\int_{W}(p_{2})^{\ast}\alpha\wedge(p_{2})^{\ast}\chi \\
&=\int_{W\setminus V}(p_{2})^{\ast}\alpha\wedge(p_{2})^{\ast}\chi \\
&=\int_{X\setminus Y}\alpha\wedge\chi \\
&=\int_{X}\alpha\wedge\chi.
\end{split}
\end{equation}
The second (resp. forth) equality is due to the facts that both of $(p_{1})^{\ast}\alpha$ and $(p_{1})^{\ast}\chi$ (resp. $\alpha$ and $\chi$) are smooth which implies the finiteness of the integral, and $V$ (resp. $Y$) is of measure zero. (\ref{e23}) exactly indicates that $(p_{1})_{\ast}(p_{1})^{\ast}\alpha=\alpha$. The rest equalities are similar.
\end{proof}
\end{lemma}

\section{The birational invariance}
\label{sec:birational}
This section is devoted to the proof of Theorems \ref{t11}, \ref{t12} and \ref{t13}. We will see that the ingredient is a topological technique developed in \cite{BCDT24}. 

\subsection{Topological preparation I}
\label{sec:topology}
Let $X$ be a simplicial complex. Then a $k$-form $\omega$ on $X$ consists of a smooth $k$-form $\omega_{\sigma}$ for every simplex $\sigma\subseteq X$ such that $\omega_{\sigma}|_{\tau}\equiv\omega_{\tau}$ whenever $\tau\subseteq\sigma$ is a subsimplex. A Riemannian metric on $X$ is a choice of a Riemannian metric $g_{\sigma}$ on every simplex $\sigma\subseteq X$ such that $g_{\sigma}|_{\tau}\equiv g_{\tau}$ for $\tau\subseteq\sigma$. Under these conventions we are able to talk about $\tilde{d}$-bounded classes on $X$.

We should recall and generalize the descriptions on hyperbolic classes in \cite{BCDT24,BKS24}. Let $V^{k}_{\textrm{asph}}(X)$ be the subspace of $H^{k}(X,\mathbb{R})$, which consists of the $k$-th cohomology classes whose pull-back to $k$-dimensional sphere must be zero. Let $V^{k}_{\textrm{hyp}}(X)$ be the subspace of $H^{k}(X,\mathbb{R})$, which consists of the $\tilde{d}$-bounded $k$-th cohomology classes. We will say that an element $[\omega]\in V^{k}_{\textrm{asph}}(X)$ (resp. $[\omega]\in V^{k}_{\textrm{hyp}}(X)$) is {\it aspherical} (resp. {\it hyperbolic}). As soon as $k\geqslant2$, every continuous map from a $k$-dimensional sphere $S^{k}\rightarrow X$ factorizes through the universal cover of $X$, so $V^{k}_{\textrm{hyp}}(X)\subseteq V^{k}_{\textrm{asph}}(X)$ for $k\geqslant2$. 

Let $G=\pi_{1}(X)$ and consider the classifying space $EG\rightarrow BG$. Given the universal cover $\tilde{X}\rightarrow X$, there is a unique (up to homotopy) classifying map of the universal cover
\[
c_{\tilde{X},X}:X\rightarrow BG
\]
such that $\tilde{X}$ is isomorphic to the pull-back $c^{\ast}_{\tilde{X},X}EG$ as a $G$-principal bundle. Define the subspace $V^{k}_{\textrm{hyp}}(BG)$ of the real singular cohomology group $H^{k}(BG,\mathbb{R})$ to be the set of $k$-cohomology classes whose pull-backs to any finite simplicial complex are hyperbolic.

Then we have
\begin{theorem}[(c.f. \cite{BKS24}, Theorem 2.5)]\label{t31}
Fix an integer $k\geqslant2$. Let $X$, $Y$ be two finite simplicial complexes such that $\pi_{i}(X)=0$ for $2\leqslant i\leqslant k-1$. Let $f:Y\rightarrow BG$ be an arbitrary continuous map, and let $[\omega]\in H^{k}(BG,\mathbb{R})$. If $c^{\ast}_{\tilde{X},X}[\omega]\in V^{k}_{\textrm{hyp}}(X)$, then $f^{\ast}[\omega]\in V^{k}_{\textrm{hyp}}(Y)$. 
\end{theorem}
Equivalently, Theorem \ref{t31} implies that if $c^{\ast}_{\tilde{X},X}[\omega]$ is hyperbolic, $[\omega]$ is also hyperbolic. We also have
\begin{lemma}\label{l31}
Fix an integer $k\geqslant2$. Then
\[
c^{\ast}_{\tilde{X},X}(H^{k}(BG,\mathbb{R}))=V^{k}_{\textrm{asph}}(X).
\]
If we furthermore assume that $\pi_{i}(X)=0$ for $2\leqslant i\leqslant k$, the pull-back 
\[
c^{\ast}_{\tilde{X},X}:H^{k}(BG,\mathbb{R})\rightarrow H^{k}(X,\mathbb{R})
\] 
is even injective.

\begin{proof}
Firstly, every class in $H^{k}(BG,\mathbb{R})$ is aspherical as is pointed out in \cite{BCDT24}, Lemma 2.9. On the other hand, the pull-back of an aspherical class via $c^{\ast}_{\tilde{X},X}$ is still aspherical by definition. It implies 
\[
c^{\ast}_{\tilde{X},X}(H^{k}(BG,\mathbb{R}))=c^{\ast}_{\tilde{X},X}(V^{k}_{\textrm{asph}}(BG))\subseteq V^{k}_{\textrm{asph}}(X).
\]

Next we construct a model of $BG$ by attaching cells of real dimension $3$ or higher to $X$ to make the universal cover contractible without affecting $\pi_{1}(X)$. In this situation, $X\subseteq BG$, and the classifying map $c_{\tilde{X},X}$ is the inclusion. Now from the long exact sequence of the pair $X\subseteq BG$ in cohomology, we get
\[
H^{k}(BG,X,\mathbb{R})\rightarrow H^{k}(BG,\mathbb{R})\xrightarrow{c^{\ast}_{\tilde{X},X}} H^{k}(X,\mathbb{R})\xrightarrow{\delta} H^{k+1}(BG,X,\mathbb{R}).
\]
It remains to prove that if $[\omega]\in V^{k}_{\textrm{asph}}(X)$, then $\delta[\omega]=0$. In fact, by the universal coefficeint theorem,  
\[
\delta[\omega]\in H_{k+1}(BG,X,\mathbb{R})^{\ast}
\] 
acts on a relative $(k+1)$-cycle $\gamma+C_{k+1}(X)$ by $\delta[\omega](\gamma+C_{k+1}(X))=\omega(\partial_{k+1}\gamma)$, where $\partial_{k+1}$ is the boundary operator. Observe that $\partial_{k+1}\gamma\in Z_{k}(X)$ is a $k$-cycle in $X$ which, if non-trivial, comes from some $(k+1)$-cells attached to $X$. It means that $\partial_{k+1}\gamma$ is a linear combination of the images of $S^{k}$. Therefore, $\omega(\partial_{k+1}\gamma)=0$ since $[\omega]$ is aspherical.

In the end, if $\pi_{i}(X)=0$ for $2\leqslant i\leqslant k$, the relative homology group $H_{k}(BG,X,\mathbb{R})$ vanishes by the relative Hurewicz theorem. So does $H^{k}(BG,X,\mathbb{R})$ by duality. Consequently, we obtain the injectivity from the exact sequence above.
\end{proof}
\end{lemma}

Combining Theorem \ref{t31} with Lemma \ref{l31}, we obtain that
\begin{corollary}\label{c31}
Fix an integer $k\geqslant2$. Let $X$ be a finite simplicial complex such that $\pi_{i}(X)=0$ for $2\leqslant i\leqslant k-1$. Then
\[
c^{\ast}_{\tilde{X},X}(V^{k}_{\textrm{hyp}}(BG))=V^{k}_{\textrm{hyp}}(X).
\]
\begin{proof}
By definition $c^{\ast}_{\tilde{X},X}(V^{k}_{\textrm{hyp}}(BG))\subseteq V^{k}_{\textrm{hyp}}(X)$. Now we prove the opposite inclusion. If $[\alpha]\in V^{k}_{\textrm{hyp}}(X)$, then $[\alpha]\in V^{k}_{\textrm{asph}}(X)$, and thus by Lemma \ref{l31} there exists a class $[\omega]\in H^{k}(BG,\mathbb{R})$ such that 
\[
[\alpha]=c^{\ast}_{\tilde{X},X}[\omega].
\] 
So $c^{\ast}_{\tilde{X},X}[\omega]$ is hyperbolic. Now, by Theorem \ref{t31}, $[\omega]\in V^{k}_{\textrm{hyp}}(BG)$. 
\end{proof}
\end{corollary}

\begin{corollary}\label{c32}
Fix an integer $k\geqslant2$. Let $f:W\rightarrow X$ be a continuous map between two finite simplicial complexes, such that $f_{\ast}:\pi_{1}(W)\rightarrow\pi_{1}(X)$ is isomorphic and $\pi_{i}(W)=0$ for $2\leqslant i\leqslant k-1$. Then
\[
f^{\ast}(V^{k}_{\textrm{hyp}}(X))=V^{k}_{\textrm{hyp}}(W).
\]
\begin{proof}
Let $\pi:\tilde{X}\rightarrow X$ be the universal cover. Since $\pi_{1}(W)\simeq\pi_{1}(X)$, $\chi:\tilde{W}:=f^{\ast}\tilde{X}\rightarrow W$ gives the universal cover of $W$. By construction, $\tilde{W}$ is the pull-back of $EG$ via $c_{\tilde{X},X}\circ f$, so that
\[
c_{\tilde{X},X}\circ f=c_{\tilde{W},W}
\]
is the classifying map of the universal cover $\tilde{W}$ as a $G$-principal bundle on $W$.

Now for an arbitrary $[\omega]\in V^{k}_{\textrm{hyp}}(W)$, by Corollary \ref{c31} there exists a class $[\alpha]\in V^{k}_{\textrm{hyp}}(BG)$ such that $c^{\ast}_{\tilde{W},W}[\alpha]=[\omega]$. Let $[\beta]=c^{\ast}_{\tilde{X},X}[\alpha]\in V^{k}_{\textrm{hyp}}(X)$. Clearly $f^{\ast}[\beta]=[\omega]$. Namely, 
\[
V^{k}_{\textrm{hyp}}(W)\subseteq f^{\ast}(V^{k}_{\textrm{hyp}}(X)).
\] 
Remember the hyperbolicity is preserved under the pull-back via a simplicial map and every continuous map is homotopic to a simplicial one. The opposite inclusion is trivial.
\end{proof}
\end{corollary}

\subsection{The balanced case}
\label{sec:balanced}
As a warm-up, we first prove the birational invariance of the balanced hyperbolicity. The ingredient is that the class of the balanced manifolds is invariant under the smooth modification, which is originally proved in \cite{AB91,AB92,AB95}. We sketch and extend their results as follows.

Let $f:\hat{X}\dashrightarrow X$ be a birational transform between compact complex manifolds of dimension $n$ with the exceptional set $E$ and the center $Y$. Let 
\[
\Gamma_{f}:=\{(x,f(x))\in\hat{X}\times X\mid x\in\hat{X}\setminus E\},
\] 
and let $\mu:W\rightarrow\bar{\Gamma}_{f}$ be the desingularization of the closure $\bar{\Gamma}_{f}$. Let $q_{1}:\bar{\Gamma}_{f}\rightarrow\hat{X}$ and $q_{2}:\bar{\Gamma}_{f}\rightarrow X$ be the natural projections. Then $p_{1}:=q_{1}\circ\mu:W\rightarrow\hat{X}$ and $p_{2}:=q_{2}\circ\mu:W\rightarrow X$ are both smooth modifications. Denote by $V$ the exceptional divisor of $\mu$. Clearly we have $p_{1}(V)\subseteq E$ and $p_{2}(V)\subseteq Y$.

Assume that $(\hat{X},\hat{\Omega})$ is balanced. By Proposition \ref{p21}, it means that for any positive $\partial\bar{\partial}$-closed $(1,1)$-current $S$ on $\hat{X}$, $\int_{\hat{X}}[\hat{\Omega}]\wedge S\geqslant0$ and $\int_{\hat{X}}[\hat{\Omega}]\wedge S=0$ if and only if $S=0$. Then we consider $[\Omega]:=(p_{2})_{\ast}(p_{1})^{\ast}[\hat{\Omega}]$. For any positive $\partial\bar{\partial}$-closed $(1,1)$-current $T$ on $X$, a direct computation implies that
\[
\int_{X}[\Omega]\wedge T=\int_{W}(p_{1})^{\ast}[\hat{\Omega}]\wedge\hat{T}=\int_{\hat{X}}[\hat{\Omega}]\wedge (p_{1})_{\ast}\hat{T}\geqslant0.
\]
Here $\hat{T}$ is the unique positive $\partial\bar{\partial}$-closed $(1,1)$-current on $W$ obtained by the following Theorem \ref{t32}.

\begin{theorem}[(c.f. \cite{AB95}, Theorem 3)]\label{t32}
Let $\mu:W\rightarrow X$ be a smooth modification between compact complex manifolds. Let $T$ be a positive $\partial\bar{\partial}$-closed $(1,1)$-current on $X$. Then there exists a unique positive $\partial\bar{\partial}$-closed $(1,1)$-current $\hat{T}$ on $W$ such that $\mu_{\ast}\hat{T}=T$ and $\hat{T}\in \mu^{\ast}[T]\in V^{1,1}(W,\mathbb{R})$.
\end{theorem}

Since $\hat{\Omega}$ is balanced, $\int_{\hat{X}}[\hat{\Omega}]\wedge (p_{1})_{\ast}\hat{T}=0$ if and only if $(p_{1})_{\ast}\hat{T}=0$. It implies that $\textrm{supp}\hat{T}\subseteq V$, hence $\textrm{supp}T=\textrm{supp}(p_{2})_{\ast}\hat{T}\subseteq Y$. Note $\textrm{codim}Y\geqslant2$, we can apply the following proposition to deduce that $T=0$. It implies that $[\Omega]$ is a balanced class on $X$ by Proposition \ref{p21}.    

\begin{proposition}[(c.f. \cite{AB92}, Theorem 1.1)]\label{p31}
Let $X$ be a complex manifold of dimension $n$. Assume $T$ is a $\partial\bar{\partial}$-closed positive $(p,p)$-current on $X$ such that the Hausdorff $2(n-p)$-measure of $\textrm{supp}\,T$ vanishes. Then $T=0$.
\end{proposition}

Now we are ready to prove the birational invariance of the balanced hyperbolic manifolds. 

\begin{proof}[Proof of Theorem \ref{t11}]
(1) Assume that $(\hat{X},\hat{\Omega})$ is balanced hyperbolic. Since the hyperbolicity is preserved under the pull-back via a continuous map, $[\Xi]:=(p_{1})^{\ast}[\hat{\Omega}]\in V^{2n-2}_{\textrm{hyp}}(W)$. Let $G=\pi_{1}(W)$, and fix a classifying space $EG\rightarrow BG$. Let $\pi:\tilde{X}\rightarrow X$ be the universal cover. Since $p_{2}$ is birational, $(p_{2})_{\ast}:\pi_{1}(W)\simeq\pi_{1}(X)$. Therefore $\chi:\tilde{W}=(p_{2})^{\ast}\tilde{X}\rightarrow W$ gives a universal cover of $W$. By construction, $\tilde{W}$ is the pull-back of $EG$ via $c_{\tilde{X},X}\circ p_{2}$, so that
\[
c_{\tilde{X},X}\circ p_{2}=c_{\tilde{W},W}
\]
is the classifying map of the universal cover $\tilde{W}$ as a $G$-principal bundle on $W$.

Since $[\Xi]\in V^{2n-2}_{\textrm{hyp}}(W)$, by Corollary \ref{c31} there exists a class $[\alpha]\in V^{2n-2}_{\textrm{hyp}}(BG)$ such that $c^{\ast}_{\tilde{W},W}[\alpha]=[\Xi]$. Let $[\Omega]=c^{\ast}_{\tilde{X},X}[\alpha]\in V^{2n-2}_{\textrm{hyp}}(X)$. Then we have $[\Xi]=(p_{2})^{\ast}[\Omega]$ by construction. Moreover, since $p_{2}$ is a birational morphism, we also have
\[
(p_{2})_{\ast}[\Xi]=(p_{2})_{\ast}(p_{2})^{\ast}[\Omega]=[\Omega].
\]
Remember that $[\Omega]=(p_{2})_{\ast}(p_{1})^{\ast}[\hat{\Omega}]$ is balanced as is shown before. It is exactly a balanced hyperbolic class on $X$.

(2) Assume that $(\hat{X},\hat{\Omega})$ is balanced hyperbolic. We apply the argument in (1) verbatim to obtain that $(p_{2})_{\ast}(p_{1})^{\ast}[\hat{\Omega}]$ is balanced hyperbolic. The converse is a mirror.
\end{proof}

The only difference between (1) and (2) is the codimension of $E$. Note that in any cases $\hat{X}$ is always a balanced manifold provided $X$ is by \cite{AB91}. However, the original proof therein doesn't work for the hyperbolicity. Hence when $f$ is a divisorial contraction and $X$ is balanced hyperbolic, it is still open whether $\hat{X}$ is a balanced hyperbolic manifold. 

\subsection{The K\"{a}hler case}
\begin{proof}[Proof of Theorem \ref{t12}]
Assume that $(\hat{X},\hat{\omega})$ is weakly K\"{a}hler $k$-hyperbolic. Since the hyperbolicity is preserved under the pull-back via a smooth map, $[\Xi]:=(p_{1})^{\ast}[\hat{\omega}^{k}]\in V^{2k}_{\textrm{hyp}}(W)$. Keep the notations as in the proof of Theorem \ref{t11}, $\chi:\tilde{W}=(q_{2})^{\ast}\tilde{X}\rightarrow W$ gives the universal cover of $W$. Moreover, $\tilde{W}$ is the pull-back of $EG$ via $c_{\tilde{X},X}\circ p_{2}$, so that
\[
c_{\tilde{X},X}\circ p_{2}=c_{\tilde{W},W}
\]
is the classifying map of the universal cover $\tilde{W}$ as a $G$-principal bundle on $W$.

Since $[\Xi]\in V^{2k}_{\textrm{hyp}}(W)$, by Corollary \ref{c31} there exists a class $[\alpha]\in V^{2k}_{\textrm{hyp}}(BG)$ such that 
\[
c^{\ast}_{\tilde{W},W}[\alpha]=[\Xi].
\] 
Let $[\beta]=c^{\ast}_{\tilde{X},X}[\alpha]\in V^{2k}_{\textrm{hyp}}(X)$. Then we have $[\Xi]=(p_{2})^{\ast}[\beta]$ by construction. Moreover, since $p_{2}$ is a birational morphism, we also have
\[
(p_{2})_{\ast}[\Xi]=(p_{2})_{\ast}(p_{2})^{\ast}[\beta]=[\beta].
\]
By hypothesis $k\mid (n-1)$, say $kt=n-1$. Then
\[
[\Xi^{t}]=(p_{1})^{\ast}[\hat{\omega}^{n-1}]\quad\textrm{and}\quad [\beta^{t}]=(p_{2})_{\ast}[\Xi^{t}].
\]
$[\Xi^{t}]$ is nef and big since $[\hat{\omega}]$ is. (See Lemma \ref{l22}.) Now we are left to prove the nefness and bigness of $[\Omega:=\beta^{t}]$. Suppose $[\Omega]$ is not nef. By (\ref{e21}) there exists a non-zero pseudo-effective $(1,1)$-class $[P]$ on $X$, such that 
\[
[\Xi^{t}]\cdot (p_{2})^{\ast}[P]=[\Omega]\cdot[P]<0.
\] 
It contradicts to the fact that $[\Xi^{t}]$ is nef and $(p_{2})^{\ast}[P]$ is pseudo-effective. Therefore $[\Omega]$ is nef.

In order to show the bigness of $[\Omega]$, we should apply the duality 
\begin{equation}\label{e31}
\overline{\mathcal{K}}=\mathcal{N}^{\vee}
\end{equation}
in \cite{BDPP13}, Theorem 2.1. Suppose $[\Omega]$ is not big. By (\ref{e31}) there exists a non-zero nef $(1,1)$-class $[\delta]$ on $X$, such that
\[
[\Xi^{t}]\cdot (p_{2})^{\ast}[\delta]=[\Omega]\cdot[\delta]<0.
\]
It contradicts to the fact that $[\Xi^{t}]$ is big and $(p_{2})^{\ast}[\delta]$ is nef. Therefore $[\Omega]$ is big.

The converse is a mirror.
\end{proof}

\subsection{On the smooth modification}
When $f$ is moreover a smooth modification, i.e $f:\hat{X}\rightarrow X$ is now a holomorphic map, we can make a more precise discussion. Note that at this time, as Jian Chen pointed it out to the authors, that a standard argument implies that we must have $\textrm{codim}E=1$.

\begin{proof}[Proof of Theorem \ref{t13}] 
(1) Assume that $(\hat{X},\hat{\Omega})$ is balanced hyperbolic. Let $G=\pi_{1}(X)$, and fix a classifying space $EG\rightarrow BG$. Since $f_{\ast}:\pi_{1}(\hat{X})\rightarrow\pi_{1}(X)$ is isomorphic, $\chi:Z:=f^{\ast}\tilde{X}\rightarrow\hat{X}$ gives the universal cover of $\hat{X}$. By construction, $Z$ is the pull-back of $EG$ via $c_{\tilde{X},X}\circ f$, so that
\[
c_{\tilde{X},X}\circ f=c_{Z,\hat{X}}
\]
is the classifying map of the universal cover $Z$ as a $G$-principal bundle on $\hat{X}$.

Since $[\hat{\Omega}]\in V^{2n-2}_{\textrm{hyp}}(\hat{X})$, by Corollary \ref{c31} there exists a class $[\alpha]\in V^{2n-2}_{\textrm{hyp}}(BG)$ such that $c^{\ast}_{Z,\hat{X}}[\alpha]=[\hat{\Omega}]$. Let $[\Omega]=c^{\ast}_{\tilde{X},X}[\alpha]\in V^{2n-2}_{\textrm{hyp}}(X)$. Then we have $[\hat{\Omega}]=f^{\ast}[\Omega]$ by construction. Moreover, since $f$ is birational, we also have
\[
f_{\ast}[\hat{\Omega}]=f_{\ast}f^{\ast}[\Omega]=[\Omega].
\]
Remember that $[\Omega]=f_{\ast}[\hat{\Omega}]$ is balanced as is shown before. It is exactly a balanced hyperbolic class on $X$. 

(2) Assume that $(X,\omega)$ is semi-balanced hyperbolic. It means that $\omega$ is non-negative and strictly positive on a Zariski open set $U$, so $f^{\ast}\omega$ is non-negative and strictly positive on a Zariski open set $f^{-1}(U)\setminus E$. The $\tilde{d}$-boundedness of $f^{\ast}\omega^{n-1}$ is directly by Lemma \ref{l23}. In summary, $f^{\ast}\omega$ is semi-balanced hyperbolic. The case of semi-K\"{a}hler $k$-hyperbolicity is similar.

(3) Assume that $(\hat{X},\hat{\beta})$ is weakly balanced $k$-hyperbolic. Then $[\hat{\beta}]\in V^{2k}_{\textrm{hyp}}(\hat{X})$. Here we use the fact that the $\tilde{d}$-boundedness of $\hat{\beta}$ implies the $d$-closedness. Similar to the proof of (1) and keep the notations there, there exists a class $[\alpha]\in V^{2k}_{\textrm{hyp}}(BG)$ such that $c^{\ast}_{Z,\hat{X}}[\alpha]=[\hat{\beta}]$. Let 
\[
[\beta]=c^{\ast}_{\tilde{X},X}[\alpha]\in V^{2k}_{\textrm{hyp}}(X).
\]  
Then $[\hat{\beta}]=c^{\ast}_{Z,\hat{X}}[\alpha]=f^{\ast}c^{\ast}_{\tilde{X},X}[\alpha]=f^{\ast}[\beta]$. Since $f$ is birational, we also have
\[
f_{\ast}[\hat{\beta}]=f_{\ast}f^{\ast}[\beta]=[\beta].
\]
By hypothesis $k\mid(n-1)$, say $kt=n-1$. Then
\[
[\hat{\beta}^{t}]=f^{\ast}[\beta^{t}]\quad\textrm{and}\quad f_{\ast}[\hat{\beta}^{t}]=[\beta^{t}].
\] 

Now we are left to prove the nefness and bigness of $[\Omega:=\beta^{t}]$. Suppose $[\Omega]$ is not nef. By (\ref{e21}) there exists a non-zero pseudo-effective $(1,1)$-class $[P]$ on $X$, such that 
\[
[\hat{\beta}^{t}]\cdot f^{\ast}[P]=[\Omega]\cdot[P]<0.
\] 
It contradicts to the fact that $[\hat{\beta}^{t}]$ is nef and $f^{\ast}[P]$ is pseudo-effective. Therefore $[\Omega]$ is nef.

Suppose $[\Omega]$ is not big. By (\ref{e31}) there exists a non-zero nef $(1,1)$-class $[\gamma]$ on $X$, such that
\[
[\hat{\beta}^{t}]\cdot f^{\ast}[\gamma]=[\Omega]\cdot[\gamma]<0.
\]
It contradicts to the fact that $[\hat{\beta}^{t}]$ is big and $f^{\ast}[\gamma]$ is nef. Therefore $[\Omega]$ is big. In summary, we obtain that $[\beta]$ is weakly balanced $k$-hyperbolic.
\end{proof}

It is possible to extend the duality (\ref{e31}) to the non-K\"{a}hler case, hence the K\"{a}hler assumption in (3) is not necessary. We will discuss this topic in an upcoming paper. Also one may wonder that if $(X,\beta)$ is weakly balanced $k$-hyperbolic, whether $\hat{X}$ will be. Naturally we could consider $f^{\ast}\beta^{t}$, which surely inherits the nefness and $\tilde{d}$-boundedness. However, as is pointed out in \cite{LX16}, Sect.5.7, a big $(n-1,n-1)$-class can be pulled back to a class on the boundary of the pseudo-effective cone. Perhaps this direction is not true.

\section{Improvement}

This section aims to prove that the class of balanced hyperbolic manifolds is invariant under some specific contractions, without limitations on the higher homotopy groups any more. In our paper, a contraction to points $f:\hat{X}\dashrightarrow X$ between compact complex manifolds of dimension $n$ is a birational transform such that the center $Y$ is a collection of points. In this setting, let $\Gamma_{f}$ be its graph, and let $W$ be the desingularization of the closure $\bar{\Gamma}_{f}$. Let $l_{1}:\Gamma_{f}\rightarrow\hat{X}$ and $l_{2}:\Gamma_{f}\rightarrow X$ be the natural projections. Let $p_{1}:W\rightarrow\hat{X}$ and $p_{2}:W\rightarrow X$ be the natural morphisms.

The proof utilizes a modification of the topological techniques developed in \cite{BKS24} and \cite{BP21}.

\subsection{Topological preparation II}
We first prove a variant of Proposition 2.3 of \cite{BP21}.

\begin{lemma}\label{l41}
Let $f:\hat{X}\dashrightarrow X$ be a contraction to points between compact complex manifolds of dimension $n$. Then for $1\leqslant i\leqslant 2n-2$, $(f|_{\hat{X}\setminus E})_\ast:\pi_{i}(\hat{X}\setminus E)\rightarrow\pi_i(X)$ are isomorphisms.
\begin{proof}
Decompose $Y$ as $Y=\cup^l_{k=1} \{p_k\}$. Let $\iota:X\setminus\{p_1\}\rightarrow X$ be the natural inclusion. Take a neighbourhood $V$, which is homeomorphic to a $2n$-cell, of $\{p_{1}\}$. Since $(X,X\setminus\{p_1\})$ is homotopic to the pair $(X,X\setminus V)$, we can instead consider the long exact sequence as follows:
\begin{equation*}
	\cdots\rightarrow\pi_{k+1}(X,X\setminus V)\rightarrow\pi_{k}(X\setminus V)\overset{\iota_\ast}{\rightarrow}\pi_{k}(X)\rightarrow
	\pi_{k}(X,X\setminus V)\rightarrow\cdots.
\end{equation*}
Note $\pi_{k}(X,X\setminus V)=0$ when $k\leqslant 2n-1$. For $1\leqslant k\leqslant 2n-2$, $\iota$ induces following isomorphisms 
\begin{equation*}
    \iota_\ast:\pi_k(X\setminus\{p_1\})\rightarrow\pi_k(X).
\end{equation*}
Inductively, we conclude that the inclusion $j:X\setminus Y\rightarrow X$ induces isomorphisms 
\[
j_\ast:\pi_k(X\setminus Y)\rightarrow\pi_k(X)\quad\textrm{for }1\leqslant k \leqslant 2n-2.
\] 
Then as $f|_{\hat{X}\setminus E}:\hat{X}\setminus E\rightarrow X\setminus Y$ is a biholomorphic map, we obtain the desired result.
\end{proof}
\end{lemma}

A simple adjustment of the proof of Lemma \ref{l41} implies the following result.
\begin{corollary}\label{c41}
Let $f:\hat{X}\dashrightarrow X$ be a birational transform between compact complex manifolds of dimension $n$, such that the center $Y$ satisfies $\dim Y=n-k-1$. Then for $1\leqslant i\leqslant 2k$, $(f|_{\hat{X}\setminus E})_\ast:\pi_{i}(\hat{X}\setminus E)\rightarrow\pi_i(X)$ are isomophisms.
\end{corollary} 
 
Now let $X$ be a smooth manifold (or more general a path-connected simplicial complex) of real dimension $2n$, and let $G=\pi_{1}(X)$. We can construct a model analogy to the classifying space $BG$ as follows. Consider the $(2n-3)$-th term $U$ of the Postnikov tower of $X$ by gluing cells of dimension $2n-1$ or higher to kill all the $i$-th homotopy groups for $i\geqslant 2n-2$. As we can furthermore attach cells of dimension $3$ or higher to obtain $BG$, $U$ is seen as the subcomplex of $BG$. Moreover, the composition of the inclusions $c_{X,n-1}:X\hookrightarrow U$ and $U\hookrightarrow BG$ is just the classifying map of the universal cover of $X$. By construction, the $(2n-2)$-skeleton of $U$ is the same as $X$. So we have  

\begin{lemma}\label{l42}
The homomorphism $c_{X,n-1}^\ast:H^{2n-2}(U,\mathbb{R})\rightarrow H^{2n-2}(X,\mathbb{R})$ is injective, and 
\begin{equation*}
	c_{X,n-1}^\ast(H^{2n-2}(U,\mathbb{R}))=V^{2n-2}_{\textrm{asph}}(X).
\end{equation*}
\begin{proof}
As $\pi_{2n-2}(U)=0$, it is clear that $V^{2n-2}_{\textrm{asph}}(U)=H^{2n-2}(U,\mathbb{R})$. Hence by definition
\[
c_{X,n-1}^{\ast}(H^{2n-2}(U,\mathbb{R}))\subseteq V^{2n-2}_{\textrm{asph}}(X).
\] 
Now consider the the long exact sequence of cohomology associated to $(X,\,U)$: 
\begin{equation*}
\cdots\rightarrow H^{2n-2}(U,X;\mathbb{R})\rightarrow H^{2n-2}(U,\mathbb{R})\xrightarrow{c_{X,n-1}^\ast} H^{2n-2}(X,\mathbb{R})\xrightarrow{\partial} H^{2n-1}(U,X;\mathbb{R})\rightarrow\cdots.
\end{equation*}
By the universal coefficient theorem for relative cohomology and the fact that $X$ and $U$ share the same $(2n-2)$-skeleton, the relative cohomology group $H^{2n-2}(U,X;\mathbb{R})$ vanishes and $c_{X,n-1}^\ast$ is injective.
 
It is left to show $V^{2n-2}_{\textrm{asph}}(X)\subseteq \text{Ker}\,\partial$. Let $[\omega]\in V^{2n-2}_{\textrm{asph}}(X)$. Then $\partial [\omega]\in H_{2n-1}(U,X;\mathbb{R})^\ast$ acts on a relative $(2n-1)$-cycle $\gamma+C_{2n-1}(X)$ by
\begin{equation*}
	\partial [\omega](\gamma+C_{2n-1}(X))=\omega(\partial_{2n-1}\gamma),
\end{equation*}
where $\partial_{2n-1}$ is boundary operator. Notice that $\partial_{2n-1}\gamma\in Z_{2n-2}(X)$ which, if non-trivial, comes from $(2n-1)$-cells attached to $X$. Thus $\partial_{2n-1}\gamma$ is a linear combination of the images of $S^{2n-2}$. As $\omega$ is an aspherical class, it follows that $\omega(\partial_{2n-1}\gamma)=0$. Namely $[\omega]\in\text{Ker}\,\partial$.
\end{proof}
\end{lemma}
If we instead construct the $(2k-1)$-th term $U_{k}$ of the Postnikov tower of $X$ by gluing cells of dimension $2k+1$ or higher to kill all the $i$-th homotopy groups for $i\geqslant 2k$, and let $c_{X,k}:X\hookrightarrow U_{k}$ be the inclusion, then a similar argument implies that
\begin{corollary}\label{c42}
The homomorphism $c_{X,k}^\ast:H^{2k}(U_{k},\mathbb{R})\rightarrow H^{2k}(X,\mathbb{R})$ is injective, and 
\begin{equation*}
	c_{X,k}^\ast(H^{2k}(U_{k},\mathbb{R}))=V^{2k}_{\textrm{asph}}(X).
\end{equation*}
\end{corollary}

If $X$ is a simplicial complex which is not necessarily finite, we can define the subspace $V^{2k}_{\textrm{hyp}}(X)$ of the real singular cohomology group $H^{2k}(X,\mathbb{R})$ to be the set of $2k$-th cohomology classes whose pull-back to any finite simplicial complex is hyperbolic. We show that $U$ possesses the same universal property for the hyperbolic $(2n-2)$-classes as $BG$.

\begin{proposition}\label{p41}
Let $X$ be a path-connected simplicial complex, and let $Y$ be a finite simplicial complex. Let $f:Y\rightarrow U$ be a continuous map, and $[\omega]\in H^{2n-2}(U,\mathbb{R})$. If $c_{X,n-1}^\ast [\omega]\in V^{2n-2}_{\textrm{hyp}}(X)$, then $f^\ast [\omega]\in V^{2n-2}_{\textrm{hyp}}(Y)$.
\begin{proof}
Take a subcomplex $X'$ of $U$, so that both $X\subseteq X'$ and $f(Y)\subseteq X'$. Remember the hyperbolicity is preserved under the pull-back of a continuous map. It is sufficient to prove that $[\omega|_{X'}]$ is a hyperbolic class on $X^{\prime}$. By definition, we need to show that for any continuous map $g:Z\rightarrow X^{\prime}$ from an arbitrary finite simplicial complex $Z$, $[\omega|_{g(Z)}]$ is hyperbolic.
	
When $g(Z)\subseteq X$, it is clear. Otherwise, $g(Z)$ is obtained by gluing finite cells of dimension $(2n-1)$ or higher to $A:=X\cap g(Z)$. Assume first the number of glued cells equals one, and denote the gluing map by
\[
h:S^{k-1}\rightarrow A\quad\textrm{with }k\geqslant 2n-1.
\] 
Then $g(Z)=A\cup_h D^k$. Let $\pi:\tilde{A}\rightarrow A$ be the universal cover of $A$, then the universal cover of $g(Z)$ is just $\tilde{A}'=\tilde{A}\cup_{h\times \pi}(D^k\times \pi_1(A))$. Let $\pi:\tilde{A}^{\prime}\rightarrow g(Z)$ be the natural projection by abuse of the notation. Choose a representative $\omega$ of $[\omega]$ in $\Omega^{2n-2}(U)$. As $\omega$ is hyperbolic on $A$, by definition, there is a bounded $(2n-3)$-form $\alpha$ on $\tilde{A}$, so that $d\alpha =\pi^\ast(\omega|_A)$. Now consider $\omega|_{D^k}$. As $H^{2n-2}_{\textrm{dR}}(D^k)=0$ for $k\geqslant 2n-1$, there is an $\alpha'\in\Omega^{2n-3}(D^k)$, such that $d\alpha'=\omega|_{D^k}$. In particular, $\alpha^{\prime}$ is also bounded since $D^{k}$ is compact. Let $h_0$ (resp. $\mathbb{D}$) be an arbitrary lift of $h$ (resp. $D^k$). Denote by $S^{k-1}_{0}$ the boundary of one sheet $D^{k}_{0}$ of $\mathbb{D}$. Then 
\begin{equation*}
d(h^\ast_0\alpha|_{S^{k-1}_{0}})=(h^\ast_0 d\alpha)|_{S^{k-1}_{0}}=h^\ast_0(\pi^{\ast}\omega)|_{S^{k-1}_{0}},
\end{equation*}  
where the commutativity of $d$ and $h^{\ast}_{0}$ is ensured by de Rham--Thom theorem for simplicial complexes. Thus 
\begin{equation*}
d((h^\ast_0\alpha-h^{\ast}_{0}\pi^\ast\alpha')|_{S^{k-1}_{0}})=h^\ast_0(\pi^\ast\omega)|_{S^{k-1}_{0}}-h^\ast_0(\pi^\ast\omega)|_{S^{k-1}_{0}}=0.
\end{equation*}
It means that $(h^\ast_{0}\alpha-h^{\ast}_{0}\pi^\ast\alpha')|_{S^{k-1}_{0}}\in \text{Ker}\,d$. As $k\geqslant 2n-1$, $H^{2n-3}_{\textrm{dR}}(S^{k-1}_{0})=0$. Thus there exists a $\xi_{0}\in\Omega^{2n-4}(S^{k-1}_{0})$, such that $d\xi_{0}=(h^\ast_{0}\alpha-h^{\ast}_{0}\pi^\ast\alpha')|_{S^{k-1}_{0}}$. The boundedness of $d\xi_{0}$ is ensured by the boundedness of $\alpha$ and $\alpha'$. Pick a smooth extension $\xi'_{0}$ of $\xi_{0}$ on $D^k_{0}$, so that $d\xi'_{0}$ is still bounded. Now we can patch $\alpha$ and $\pi^{\ast}\alpha^{\prime}+d\xi^{\prime}_{0}$ together to obtain an $\hat{\alpha}_{0}$ on $\tilde{A}\cup_{h_{0}}D^{k}_{0}$ such that $\pi^{\ast}\omega=d\hat{\alpha}_{0}$ restricted on this space. Following this streamline, we obtain an $\hat{\alpha}$ on the whole $\tilde{A}^{\prime}$ such that $\pi^{\ast}\omega=d\hat{\alpha}$. Since $\hat{\alpha}$ is bounded, $[\omega|_{g(Z)}]$ is hyperbolic.
	 
When the number of glued cells is larger, we inductively obtain the desired result.
\end{proof}
\end{proposition}
It naturally extends to the following form.
\begin{corollary}\label{c43}
Let $X$ be a path-connected simplicial complex, and let $Y$ be a finite simplicial complex. Let $f:Y\rightarrow U_{k}$ be a continuous map, and $[\omega]\in H^{2k}(U_{k},\mathbb{R})$. If $c_{X,k}^\ast [\omega]\in V^{2k}_{\textrm{hyp}}(X)$, then $f^\ast [\omega]\in V^{2k}_{\textrm{hyp}}(Y)$.
\end{corollary}

Now we have
\begin{proposition}\label{p42}
Let $X$ be a path-connected simplicial complex, then 
\[
c_{X,n-1}^{\ast}(V^{2n-2}_{\textrm{hyp}}(U))=V^{2n-2}_{\textrm{hyp}}(X).
\]
\begin{proof}
By definition, $c_{X,n-1}^{\ast}(V^{2n-2}_{\textrm{hyp}}(U))\subseteq V^{2n-2}_{\textrm{hyp}}(X)$. On the other hand, if $[\omega]\in V^{2n-2}_{\textrm{hyp}}(X)$, then naturally $[\omega]\in V^{2n-2}_{\textrm{asph}}(X)$. By Lemma \ref{l42}, there exists a class $[\alpha]\in H^{2n-2}(U,\mathbb{R})$, so that $[\omega]=c_{X,n-1}^{\ast}[\alpha]$. Then by Proposition \ref{p41} the pull-back of $[\alpha]$ to any finite simplicial complex is hyperbolic. Therefore $[\alpha]\in V^{2n-2}_{\textrm{hyp}}(U)$.
\end{proof}
\end{proposition}

\begin{corollary}\label{c44}
Let $f:W\rightarrow X$ be a continuous map between two path-connected simplicial complexes, such that $f_{\ast}:\pi_{i}(W)\rightarrow\pi_{i}(X)$ is isomorphic for $1\leqslant i\leqslant 2n-3$. Then
\[
f^{\ast}(V^{2n-2}_{\textrm{hyp}}(X))=V^{2n-2}_{\textrm{hyp}}(W).
\]
\begin{proof}
Let $\pi:\tilde{X}\rightarrow X$ be the universal cover, and let $U$ be the $(2n-3)$-th term of the Postnikov tower of $X$ constructed before. Since $\pi_{1}(W)\simeq\pi_{1}(X)$, $\chi:\tilde{W}:=f^{\ast}\tilde{X}\rightarrow W$ gives a universal cover of $W$. Moreover, since $\pi_{i}(W)\simeq\pi_{i}(X)$ for $1\leqslant i\leqslant 2n-3$, by the uniqueness (up to the homotopy) of the Postnikov tower (see \cite{Hat02}), $c_{W,n-1}:W\rightarrow U$ is also the $(2n-3)$-th term of the Postnikov tower of $W$, and we have the homotopy equivalence $c_{X,n-1}\circ f\simeq c_{W,n-1}$. 

Now for an arbitrary $[\omega]\in V^{2n-2}_{\textrm{hyp}}(W)$, by Proposition \ref{p42} there exists a class $[\alpha]\in V^{2n-2}_{\textrm{hyp}}(U)$ such that $c^{\ast}_{W,n-1}[\alpha]=[\omega]$. Let $[\beta]=c^{\ast}_{X,n-1}[\alpha]\in V^{2n-2}_{\textrm{hyp}}(X)$. Clearly $f^{\ast}[\beta]=[\omega]$. Namely, $V^{2n-2}_{\textrm{hyp}}(W)\subseteq f^{\ast}(V^{2n-2}_{\textrm{hyp}}(X))$. Remember the hyperbolicity is preserved under the pull-back via a continuous map. The opposite inclusion is trivial.
\end{proof}
\end{corollary}

Replacing Lemma \ref{l42} and Proposition \ref{p41} by Corollaries \ref{c42} and \ref{c43}, we obtain
\begin{corollary}\label{c45}
\begin{enumerate}
    \item[(1)] $c_{X,k}^{\ast}(V^{2k}_{\textrm{hyp}}(U_{k}))=V^{2k}_{\textrm{hyp}}(X)$.
    \item[(2)] Assume that $f:W\rightarrow X$ is a continuous map between two path-connected simplicial complexes, such that $f_{\ast}:\pi_{i}(W)\rightarrow\pi_{i}(X)$ is isomorphic for $1\leqslant i\leqslant 2k-1$. Then
    \[
      f^{\ast}(V^{2k}_{\textrm{hyp}}(X))=V^{2k}_{\textrm{hyp}}(W).
    \]
\end{enumerate}

\end{corollary}

\subsection{On the contraction to points}
We are in a good position to prove Theorems \ref{t14} and \ref{t15}. In particular, the vanishing requirement for the higher homotopy groups is not necessary any more.  

\begin{proof}[Proof of Theorem \ref{t14}]
Assume that $(\hat{X},\hat{\Omega})$ is balanced hyperbolic. Let $G=\pi_1(X)$, and $\tilde{X}$ is the universal covering of $X$. Keep the notations before, we should consider the following commutative diagram:
\begin{equation*}
\xymatrix@C=2pc@R=2pc{
& U\\
\Gamma_{f}\ar[r]^{l_{2}}\ar[ur]^{c_{\Gamma_{f},n-1}\simeq c_{X,n-1}\circ l_{2}} & X\ar[u]_{c_{X,n-1}}
}
\end{equation*}
 
By Lemma \ref{l41}, for $1\leqslant i\leqslant 2n-3$, $(l_{2})_\ast:\pi_{i}(\Gamma_{f})\simeq\pi_{i}(\hat{X}\setminus E)\rightarrow\pi_{i}(X)$ and $(c_{X,n-1})_\ast:\pi_{i}(X)\rightarrow\pi_{i}(U)$ are isomorphisms, while for $i\geq 2n-2$, $\pi_i(U)=0$. Then by the uniqueness (up to the homotopy) of the Postnikov tower of $\Gamma_{f}$, $c_{\Gamma_{f},n-1}:\Gamma_{f}\rightarrow U$ is also the $(2n-3)$-th term of the Postnikov tower of $\Gamma_{f}$, and we have the homotopy equivalence $c_{X,n-1}\circ l_{2}\simeq c_{\Gamma_{f},n-1}$. Moreover, as $U$ can be regarded as a subcomplex of $BG$, the universal cover $Z:=f^{\ast}\tilde{X}$ of $\Gamma_{f}$ is actually the pull-back of $EG|_U$ by $c_{X,n-1}\circ l_{2}$.

As $[\hat{\Omega}]\in V^{2}_{\textrm{hyp}}(\hat{X})$, $[\Xi]:=(l_{1})^{\ast}[\hat{\Omega}]\in V^{2}_{\textrm{hyp}}(\Gamma_{f})$. By Proposition \ref{p42} there exists a unique class $[\alpha]\in V^{2}_{\textrm{hyp}}(U)$, so that $c_{\Gamma_{f},n-1}^{\ast}[\alpha]=[\Xi]$. Denote $[\beta]=c_{X,n-1}^{\ast}[\alpha]\in V^{2n-2}_{\textrm{hyp}}(X)$, thus $[\Xi]=(l_{2})^\ast[\beta]$. Moreover, as $l_2$ is a birational morphism, 
\[
(l_{2})_\ast[\Xi]=(l_{2})_{\ast}(l_{2})^{\ast}[\beta]=[\beta].
\] 
Remember that 
\[
[\beta]=(l_{2})_{\ast}(l_{1})^{\ast}[\hat{\Omega}]=(p_{2})_{\ast}(p_{1})^{\ast}[\hat{\Omega}]
\] 
is balanced as is shown before. It is exactly a balanced hyperbolic class on $X$.
\end{proof}

Since on a complex surface, a balanced metric must be a K\"{a}hler metric, Theorem \ref{t14} implies that, for a contraction to the points $f:\hat{X}\dashrightarrow X$ between compact complex surfaces, if $\hat{X}$ is K\"{a}hler hyperbolic, $X$ is also K\"ahler hyperbolic. 

\begin{proof}[Proof of Theorem \ref{t15}]
$(l_{2})^{\ast}(V^{2k}_{\textrm{hyp}}(X))=V^{2k}_{\textrm{hyp}}(\Gamma_{f})$ is due to Corollary \ref{c45}.

Now for an arbitrary $[\omega]\in V^{2k}_{\textrm{hyp}}(X)$, $(l_{2})^{\ast}[\omega]\in V^{2k}_{\textrm{hyp}}(\Gamma_{f})$ and $(l_{2})_{\ast}(l_{2})^{\ast}[\omega]=[\omega]$. Thus 
\[
V^{2k}_{\textrm{hyp}}(X)\subseteq (l_{2})_{\ast}(V^{2k}_{\textrm{hyp}}(\Gamma_{f})).
\]

On the other hand, if $[\Xi]\in V^{2k}_{\textrm{hyp}}(\Gamma_{f})$, by Corollary \ref{c45} there exists an $[\alpha]\in V^{2k}_{\textrm{hyp}}(U_{k})$, so that $c_{\Gamma_{f},k}^{\ast}[\alpha]=[\Xi]$. Denote $[\beta]=c_{X,k}^{\ast}[\alpha]\in V^{2k}_{\textrm{hyp}}(X)$, thus $[\Xi]=(l_{2})^\ast[\beta]$. Moreover, as $l_{2}$ is a birational morphism, 
\[
(l_{2})_\ast[\Xi]=(l_{2})_\ast(l_{2})^\ast[\beta]=[\beta].
\] 
In summary, we obtain that $(l_{2})_{\ast}(V^{2k}_{\textrm{hyp}}(\Gamma_{f}))=V^{2k}_{\textrm{hyp}}(X)$. 

In the end, assume that $(\hat{X},\hat{\omega})$ is weakly K\"{a}hler $k$-hyperbolic. Then
\[ 
[\gamma]:=(l_{2})_{\ast}(l_{1})^{\ast}[\hat{\omega}^{k}]\in V^{2k}_{\textrm{hyp}}(X).
\]  
By hypothesis $k\mid(n-1)$, say $kt=n-1$. We are left to prove the nefness and bigness of $[\Omega:=\gamma]$. 

Suppose $[\Omega]$ is not nef. By (\ref{e21}) there exists a non-zero pseudo-effective $(1,1)$-class $[P]$ on $X$, such that 
\[
0>[\Omega]\cdot[P]=(p_{2})_{\ast}(p_{1})^{\ast}[\hat{\omega}^{n-1}]\cdot [P]=(p_{1})^{\ast}[\hat{\omega}^{n-1}]\cdot(p_{2})^{\ast}[P].
\] 
It contradicts to the facts that $(p_{1})^{\ast}[\hat{\omega}^{n-1}]$ is nef and $(p_{2})^{\ast}[P]$ is pseudo-effective. Therefore $[\Omega]$ is nef.

Suppose $[\Omega]$ is not big. By (\ref{e31}) there exists a non-zero nef $(1,1)$-class $[\delta]$ on $X$, such that
\[
0>[\Omega]\cdot[\delta]=(p_{2})_{\ast}(p_{1})^{\ast}[\hat{\omega}^{n-1}]\cdot[\delta]=(p_{1})^{\ast}[\hat{\omega}^{n-1}]\cdot(p_{2})^{\ast}[\delta].
\]
It contradicts to the facts that $(p_{1})^{\ast}[\hat{\omega}^{n-1}]$ is big and $(p_{2})^{\ast}[\delta]$ is nef. Therefore $[\Omega]$ is big.
\end{proof}

\begin{acknowledgements}
The authors are grateful to Prof. Jian Chen for reading an earlier version and his helpful remarks.
\end{acknowledgements}

\end{document}